\documentclass[10pt]{amsart}
\usepackage{amsxtra,amssymb,amsthm,amsmath,latexsym}

\theoremstyle{plain}
\newtheorem{theorem}{Theorem}[section]
\newtheorem{definition}[theorem]{Definition}

\newtheorem{lemma}[theorem]{Lemma}

\numberwithin{equation}{section}

\begin{document}

\newcommand{\coloneqq}{\mathrel{\raise.095ex\hbox{:}\mkern-4.2mu=}}
\def\ecolon{\mathrel{\mathop:}=}

\title[Uniqueness for the two-phase Stefan problem] {Uniqueness for solutions of the two-phase Stefan 
problem with signed measures as data}

\author{Marianne K. Korten}
\address{Department of Mathematics \\
Kansas State University \\
Manhattan, KS 66506}
\email{marianne@math.ksu.edu}

\author{Charles N. Moore}
\email{cnmoore@math.ksu.edu}


\begin{abstract}
We consider the two-phase Stefan problem $u_t=\Delta\alpha(u)$ where 
$\alpha(u) =u+1$ for $u<-1$, $\alpha(u) =0$ for $-1 \leq u \leq 1$,
and $\alpha(u)=u-1$ for $u > 1$.  We show uniqueness of solutions which have signed measures as initial data,
that is, we show that if the 
difference of two 
solutions $u$ and 
$v$ defined 
on $\mathbb R^n \times (0,T)$ vanishes in a weak sense as $t \to 0$ then
$u=v$ a.e.  
\end{abstract}

\maketitle

\section{Introduction}\label{S:1.}
In this  paper we show uniqueness theorems for distributional
solutions to the degenerate parabolic equation
\begin{equation}\label{classical}
\frac{\partial u}{\partial t} =\Delta \alpha(u)
\end{equation}
in the domain $\mathbb R^n \times (0,T)$, for some $T>0$.
Here $\alpha(u)=0$ if $-1 \leq u \leq 1$, $\alpha(u)=u-1$ for $u>1$,
and $\alpha(u)=u+1$ for $u<-1.$

Equation \eqref{classical} is known as the two-phase Stefan problem, which describes the
flow of heat within a substance which can be in a liquid phase or a solid phase, and for which
there is a latent heat to initiate phase change. This allows for a ``mushy zone", that is, a 
region which is between the liquid and solid phases.
In this model, $u$ represents the enthalpy
and $\alpha(u)$ the temperature.  

Throughout, we will consider distributional solutions to \eqref{classical}, that is, we consider  $u \in 
L^1_{loc}(\mathbb{R}^n\times (0,T))$  which satisfy
\begin{equation}\label{distsolns}
\int_0^T \int_{\mathbb R^n} \alpha(u) \Delta \varphi + u \varphi_t dx \, dt =0
\end{equation}
for every $\varphi \in C^{\infty}$ with compact support in $\mathbb R^n 
\times (0,T).$

For non-negative solutions of \eqref{distsolns}, a collection of techniques is 
available. It is possible to identify the growth at infinity
and to show
existence and uniqueness of solutions to the Cauchy problem in the
optimal class (of growth at infinity) for measure data. See Korten 
\cite{Ko} and Andreucci and Korten  \cite{AKo}. In this present paper,  
the fact that we are working with signed solutions
complicates matters and we will need to devise a different strategy.

Previous work of the authors (Korten and Moore \cite{KoM1}) 
provides a solution for the Cauchy problem for equation \eqref{distsolns}.

\begin{definition}  Given a Radon measure $\mu$ with $\int_{\mathbb R^n} \exp(-c|x|^2) d|\mu|< \infty$ for some $c>0,$
we say that $u$ which is integrable on bounded subsets of $\mathbb R^n \times (0,T)$
satisfies the Cauchy problem with initial data $\mu$ for the two-phase Stefan problem in the sense of conservation laws if 
\begin{equation*}
\int_0^T \int_{\mathbb R^n} \alpha(u)\Delta\varphi + u \varphi_t dx dt + \int_{\mathbb R^n} \varphi(x,0) d\mu=0
\end{equation*}
for every $\varphi \in C^{\infty}(\mathbb R^n \times (-\infty,T))$ with compact support.
\end{definition}

\begin{theorem}\cite{KoM1} Suppose $\mu$ is a Radon measure which satisfies $\int_{\mathbb R^n} \exp(-c|x|^2) d|\mu| < \infty$
for some $c>0.$ Then there exists a solution of the Cauchy problem with initial data $\mu$ for the two-phase Stefan problem in the sense of conservation laws on $\mathbb R^n \times (0,T),$ where 
$T=\frac{1}{4c}.$
\end{theorem}

In \cite{KoM1} it is shown that a solution in the sense of conservation laws is a solution in the sense of distributions.  Consequently, we have:

\begin{theorem}\cite{KoM1}  Suppose $\mu$ is a Radon measure satisfying $\int_{\mathbb R^n} \exp(-c|x|^2) d|\mu| < \infty$
for some $c>0.$ Then there exists a solution of \eqref{distsolns} on  $\mathbb R^n \times (0,T),$ $T=\frac{1}{4c},$ 
which has initial value $\mu$ in the sense 
that $\lim_{t \to 0} \int_{\mathbb R^n} u(x,t) \psi(x) dx =
\int_{\mathbb R^n}  \psi(x) d\mu $ for every function $\psi \in C^{\infty}(\mathbb R^n)$ of compact support.
\end{theorem}

Note that a solution in the sense of conservation laws is integrable on bounded subsets of $\mathbb R^n \times (0,T),$
whereas a distributional solution is required only to be integrable on compact subsets of  $\mathbb R^n \times (0,T);$
for $u$ which are integrable on bounded subsets of  $\mathbb R^n \times (0,T),$ the two notions of solution are equivalent.

In other previous work, the authors have discussed regularity of solutions.  In Korten and Moore \cite{KoM2}
the authors show intrinsic energy estimates for
signed local solutions to \eqref{distsolns}. These estimates do not involve initial or boundary
data. Subcaloric estimates are then used to show that if
$u \in L^2_{loc}$ is a solution of the two-phase Stefan problem, then $\alpha(u)$ is locally
bounded.  It then follows from a theorem of Caffarelli and Evans \cite{CE} (or similar results
by Sacks \cite{Sa}, Ziemer \cite{Z}, or DiBenedetto \cite{D}) that $\alpha(u)$ is continuous.

Our purpose in this paper is to show the following uniqueness theorem.

\begin{theorem} \label{maintheorem}  
Suppose $u$ and $v$ are  solutions of
\eqref{distsolns}   on $\mathbb R^n \times (0,T)$ which
belong to $L^{\infty}(\mathbb R^n \times (\varepsilon, T))$ for 
every $\varepsilon >0,$ and which satisfy 
\begin{equation*} 
\int_{\mathbb R^n} \int_0^T (|u(x,t)| + |v(x,t)|) e^{-c|x|^2} dt dx < \infty, 
\end{equation*}
for some $c>0.$  
If  for every $\varphi \in C^{\infty}(\mathbb R^n)$  of compact 
support we have
\begin{equation*}
\lim_{t \to 0} \int_{\mathbb R^n} (u(x,t)-v(x,t)) \varphi(x) dx =0,
\end{equation*}
then $u(x,t)=v(x,t)$ a.e. on $\mathbb R^n \times (0,T).$
\end{theorem}

The hypotheses of the theorem essentially say that $u=v$ initially in the sense of measures. 
The novelty of our result is that it allows for signed solutions.
Pierre \cite{P} demonstrated a similar uniqueness result for nonnegative solutions $u$  to 
equations of the 
form $u_t- \Delta \varphi(u)=0,$ where $\varphi$ is assumed to be nondecreasing, locally Lipschitz, and with 
$\varphi(0)=0$ and $u \in L^1(\mathbb R^n \times(0,T)) \cap  L^{\infty}(\mathbb R^n \times (\varepsilon, T)),$ for every $\varepsilon >0.$  
Bouillet \cite{B} considers solutions $u$  (possibly of changing 
sign)
of $u_t- \Delta \alpha (u) =0,$ which satisfy the same growth conditions as in our
theorem, with Lipschitz $\alpha(u)$, and  obtains uniqueness under the stronger hypothesis that
$\|u(\cdot,t)-v(\cdot,t)\|_{L^1_{loc}(\mathbb R^n)} \to 0,$ as $t \to 0,$  when $n \geq 2.$  For $n=1,$  
Bouillet has obtained the result above.  
Related to the two-phase Stefan problem is the porous medium equation $u_t=\Delta u^m,$ $m>1.$
This has been studied extensively by many authors, including uniqueness results for nonnegative solutions.

We do not know optimal conditions under which a solution $u$ belongs to  $ L^{\infty}(\mathbb R^n \times (\varepsilon, T))$ for 
every $\varepsilon >0.$
Korten and Moore \cite{KoM2} show that the assumption $u \in L^2_{loc}(\mathbb R^n \times (0,T))$ 
implies $\alpha(u)$ is continuous, 
hence $u$ is locally bounded.  For nonnegative solutions, Korten \cite{Ko} shows that if $u\in L^1_{loc}$ is a solution of \eqref{distsolns} then $\alpha(u)$ is
continuous, but this is an open question for solutions which change sign.

We  remark 
that on a fixed compact set $K \subset \mathbb R^n$, the functionals
$T_t(\varphi(x)) = \int_{K} (u(x,t)-v(x,t)) \varphi(x) dx,$  for $\varphi$ supported in $K$, have 
$\lim_{t \to 0} T_t(\varphi)= 0.$  In particular, $T_t(\varphi)$ is bounded for fixed $\varphi$, so by the principle of
uniform boundedness, $\|T_t\| =\|u(\cdot,t) - v(\cdot, t)\|_{L^1(K)}$ is uniformly bounded.

The paper is structured as follows: In section 2 we establish some technical lemmas
which will be used in the proof of Theorem 1.4, which we prove in section 3.  Section 4
contains some further discussion.
Throughout, the letter $C$ will denote a constant which may vary from line to line.

Marianne Korten is supported by grant \#0503914 from the National Science Foundation.

\section{A few technical estimates}\label{S:2.}

The estimates in this section involve elementary computations with the Gaussian kernel
and the maximum principle.  Throughout, $B(R)$ denotes the ball centered at $0$ of radius $R.$
Some of the ideas in this section are taken from the previously mentioned work of 
Bouillet \cite{B}.

\begin{lemma}\label{boundfunction}  Consider $[1,R] \subset \mathbb R$ and on $[1,R] \times [0, \infty)$ 
let $w(x,t)$ be the solution to  $w_t =\Delta w,$ with initial condition 
$w(x,0)=0,$ and boundary conditions $w(1,t)=1$  and $w(R,t)=0$ for all $t>0.$  Then
\item(i) $\frac{\partial w}{\partial t} >0$ for all $x \in (1,R)$, $t>0$
\item(ii) $ \left|\frac{\partial w}{\partial x} (R,t) \right| \leq \exp(\frac{-R^2}{8t})$ 
whenever $0 \leq t \leq T$ if $R$ is sufficiently large.

\end{lemma}

\begin{proof} Observe that $\frac{\partial w}{\partial t}$ satisfies the heat equation with
boundary values $\frac{\partial w}{\partial t}(1,t)=\frac{\partial w}{\partial t}(R,t)=0$ for $t>0$
and initial values $\frac{\partial w}{\partial t}(x,0) >0.$ Then (i) follows immediately from the maximum principle. 
For (ii) consider the function $\tilde w(x,t)$ on $(-\infty, R] \times (0,\infty)$ 
which satisfies $\tilde w_t = \Delta \tilde w,$ $\tilde w(x,0) =4$ if 
$x \leq 1,$ $ \tilde w(x,0)=0$ if $1<x \leq R,$ and $\tilde w (R,t)= 0$ for all $t>0.$
This can be given explicitly as:
\begin{equation*}
\tilde w(x,t)= \frac{4}{\sqrt{4\pi t}} \int_{-\infty}^1 e^{\frac{-(x-s)^2}{4t}}- e^{\frac{-(x+s-2R)^2}{4t}} ds
\end{equation*}
Now 
\begin{equation*}
\tilde w(1,t)= \frac{4}{\sqrt{4\pi t}} \int_{-\infty}^1 e^{\frac{-(1-s)^2}{4t}}ds
-\frac{4}{\sqrt{4\pi t}} \int_{-\infty}^1 e^{\frac{-(1+s-2R)^2}{4t}} ds
= 2- \frac{4}{\sqrt{\pi}}\int_{-\infty}^{\frac{1-R}{\sqrt t}}  e^{-u^2}du. 
\end{equation*}
If $R$ is chosen sufficiently large, then for  $t <  T$ the last integral is bounded by $\frac{\sqrt{\pi}}{4}$.
Thus, for $t< T,$ $\tilde w (1,t) > 1.$  So by the maximum principle, 
$w(x,t) \leq \tilde w(x,t)$ for $1<x<R$ and $0<t< T .$  Since $w(R,t)=\tilde w(R,t)=0$
then $|\frac{\partial w}{\partial x}(R,t)| \leq |\frac{ \partial \tilde w}{\partial x}(R,t)|.$ Explicitly differentiating we find
\begin{equation*}
\left|\frac{\partial \tilde w}{\partial x}(R,t)\right| = 
\left|\frac{4}{\sqrt{4\pi t}} \int_{-\infty}^1 \frac{s-R}{t} e^{\frac{-(s-R)^2}{4t}} ds \right|
=\left|\frac{4}{\sqrt{\pi}} \int_{-\infty}^{\frac{1-R}{2\sqrt{t}}} \frac{2y}{\sqrt{t}} e^{-y^2} dy \right| 
= \left|\frac{4}{\sqrt{\pi t}} e^{\frac{-(1-R)^2}{4t}} \right|. 
\end{equation*}
If  $R$ sufficiently large, this last quantity is $
\leq \frac{4}{\sqrt{\pi}} e^{\frac{-R^2}{8t}} $ for all $0 < t <T.$

\end{proof}

\begin{lemma}\label{normalderivative} Suppose $0 \leq f \leq 1$ is supported in the ball $B(1) \subset \mathbb R^n,$  $d(x,t)$ is a 
smooth, bounded
function on $B(R) \times (0,T),$ and $\kappa \leq d(x,t) \leq 1$ for some $\kappa >0.$
Suppose $R>1$ is sufficiently large (as large as required for (ii) in the previous lemma)  and let $h(x,t)$ satisfy
\begin{equation*}
\begin{aligned}
d(x,t) \Delta h(x,t)&=\frac{\partial h}{\partial t}(x,t) &\hbox{ \ on\ } B(R) \times (0,T), \\
h(x,0)&=f(x), & x \in B(R),\\
h(x,t)&=0,  &  (x,t) \in \partial B(R) \times (0,T).
\end{aligned}
\end{equation*}
Then for $(y,t)$ with $|y|=R,$ $t\in (0,T),$
\begin{equation}\label{hbounded}
\left|\frac{\partial h}{\partial n}(y,t)\right| \leq  C\exp\left(\frac{-R^2}{8t}\right).
\end{equation}
Here $n$ denotes the outward normal to the ball $B(R)$ and $C$ is an absolute constant (in particular, it's independent of $R$, $\kappa$ and $d(x,t)$).
\end{lemma}

\begin{proof}  Consider $(y,t)$ where $y$ is of the form $y=(R, 0,\dots,0).$ By rotation,
it suffices to show \eqref{hbounded} for such $(y,t).$ 

Set $D=B(R) \cap \{x \in \mathbb R^n : x_1 >1\}.$ On $D \times [0,\infty)$ set $W(x_1, \dots, x_n ,t)= w(x_1,t),$
where $w$ is as in Lemma \ref{boundfunction}. Then $\Delta W = W_t.$  By the maximum principle, $|h(x,t)| \leq 1$ on $B(R) \times (0,T)$, so in
particular, $W(x,t)\geq h(x,t)$ whenever $x_1 =1.$ On $(\partial B(R) \cap \{(x,t): x_1 >1 \}) \times (0,T),$
$W(x,t) \geq 0 =h(x,t).$  Consequently, $h(x,t) - W(x,t) \leq 0$ on
$\partial D \times (0,T).$ When $t=0,$ $x \in D,$ $W(x,t)=0=h(x,t).$ Thus, $h- W \leq 0$ on the parabolic boundary of $D \times (0,T).$   
Furthermore, on $D \times (0,T),$ since $d \leq 1,$ and $\Delta W \geq 0$ by  Lemma \ref{boundfunction}(i), then 
\begin{equation*}
d\Delta(h-W)=d\Delta h -d \Delta W  \geq d  \Delta h - \Delta W = h_t- W_t.
\end{equation*} 
Therefore, by the maximum principle, $h(x,t) \leq W(x,t)$ on $D \times (0,T).$ 
In particular,
for  $x=(x_1, 0,\dots, 0),$ $1<x_1<R,$ $h(x,t) \leq W(x,t).$  This combined with $0=h(R,0, \dots, 0,t) =W(R,0, \dots,0,t)$ gives
\begin{equation*}
\left|\frac{\partial h}{\partial n}(R,0,\dots,0,t)\right|=\left|\frac{\partial h}{\partial x_1}(R,0,\dots,0,t) \right| 
\leq \left|\frac{\partial W}{\partial x_1} (R,0,\dots,0,t)\right| =\left|\frac{\partial w}{\partial x}(R,t)\right| \leq 
C \exp\left(\frac{-R^2}{8 t}\right).
\end{equation*}

\end{proof}

\begin{lemma}\label{representingformula} Let $u$  be a solution of 
\eqref{distsolns} such that $u\in L^{\infty}(\mathbb R^n \times (\varepsilon, T))$ for every $\varepsilon >0.$
Then for a.e.  $0 \leq t_1 <t_2 <T,$ and a.e. $R>0$,
\begin{equation*}
\begin{aligned}
\int_{B(R)}u(x,t_2) \phi(x,t_2) dx&=\int_{B(R)}u(x,t_1) \phi(x,t_1) dx - \int_{t_1}^{t_2} \int_{\partial B(R)} \alpha(u) \frac{\partial 
\phi}{\partial n} d\sigma dt \\
& \qquad + \int_{t_1}^{t_2} \int_{B(R)} u \frac{\partial \phi}{\partial t}  + \alpha(u)\Delta \phi \, dxdt
\end{aligned}
\end{equation*}
for every  $\phi(x,t)\in C^2(B(R)\times (t_1,t_2)) \cap C(\overline{B(R)} \times [t_1,t_2])$  which vanishes on $\partial B(R) \times 
[t_1, 
t_2].$  Here $n$ is the outward normal to 
$\partial B(R).$ 
\end{lemma}

\begin{proof}
Let $\varphi(y,s)$ be a smooth radial function of compact support on $\mathbb R^{n+1}$ and for $m=1, 2, 3,   \dots$ set 
$\varphi_m(y,s)=m\varphi (my, ms).$   For $(x,t) \in \mathbb R^n \times (0,T)$ and 
$m$ 
sufficiently large (depending on $(x,t)),$ $\varphi_m (x-y, t-s)$ is a test function and thus
\begin{equation}\label{convolution}
\int_0^T \int_{\mathbb R^n} -u(y,s) \frac{\partial \varphi_m}{\partial t}(x-y, t-s) + \alpha(u(y,s)) \Delta \varphi_m(x-y, t-s) dy ds 
=0.
\end{equation}
Set $u_m= \varphi_m \ast u$ and $w_m= \varphi_m \ast \alpha(u).$  Then on compact subsets of $\mathbb R^n \times (0,T)$, if $m$ is 
sufficiently large, we can rewrite \eqref{convolution} as the pointwise equality $-\frac{\partial}{\partial t}u_m + \Delta w_m =0.$

Since $u$ and $\alpha(u)$ are locally integrable, then a.e.~ point of $\mathbb R^n \times (0,T)$ is a Lebesgue point of both functions.
By Fubini's theorem, for a.e.~ $t,$  $0<t<T,$ a.e.~ point of $\mathbb R^n \times\{t\}$ is a Lebesgue point of $u$ and $\alpha(u).$ 
Similarly, for a.e.~ $R$, a.e.~ point of $\partial B(R) \times (0,T)$ is a Lebesgue point of $u$ and $\alpha(u).$ Consider $0<t_1<t_2 
<T$, with the property that a.e.~ point of $\mathbb R^n \times \{t_1\}$ and $\mathbb R^n \times \{t_2\}$ is a Lebesgue point of $u$ and 
$\alpha(u)$, and consider $R>0$ with the property that a.e.~ point of $\partial B(R) \times (0,T)$ is a Lebesgue point of $u$ and 
$\alpha(u)$.  Then for sufficiently large $m,$ and with  $\phi$ as in the hypotheses, we have:
\begin{equation*}
\begin{aligned}
0&= \int_{t_1}^{t_2} \int_{B(R)} \phi \frac{\partial u_m}{\partial t}  - \phi \Delta w_m \,  dx dt \\
&= \int_{B(R)} u_m(x,t_2)\phi(x,t_2)dx -\int_{B(R)} u_m(x,t_1) \phi(x,t_1) dx -\int_{t_1}^{t_2} \int_{B(R)} u_m \frac{\partial 
\phi}{\partial t}  dx dt \\
& \qquad +\int_{t_1}^{t_2} \int_{\partial B(R)} \frac{\partial \phi}{\partial n} w_m d\sigma dt - \int_{t_1}^{t_2}\int_{B(R)} w_m \Delta 
\phi \, dx dt. 
\end{aligned}
\end{equation*}
Let $m \to \infty$ and rearrange to obtain the conclusion of the lemma.
\end{proof}

In all subsequent uses of this lemma, we will assume that the $t_1$ and $t_2$ as well as the $R$ are chosen in the sets of full measure
for which the above formula is valid.

\section{The proof of the Theorem}

Fix $\Theta(x) \in C^{\infty}(\mathbb R^n)$ of compact support.   Consider $t_0$ with
$t_0 < \min \{\frac{1}{8c}, T \}.$ 
We will show that for a.e. such $t_0,$ $\int_{\mathbb R^n} (u(x,t_0)-v(x,t_0)) \Theta(x) dx=0;$
that this is then true for a.e. $t_0,$ $0 <t_0 <  \min\{\frac{1}{8c}, T\}$
 and such $\Theta$ implies $u=v$ a.e.  on $\mathbb R^n \times (0,  \min\{\frac{1}{8c}, T\}).$ 
If $\frac{1}{8c} < T,$ then note that the hypotheses will now hold on 
$(\frac{1}{16c}, \min\{\frac{3}{16c},T \})$ and repeating the argument gives that $u=v$ a.e.
on $(0, \min \{\frac{3}{16c},T \}).$ Continuing, eventually we obtain $u=v$ a.e. on $\mathbb R^n \times (0,T).$

Without loss of generality we will assume that $\Theta$ is supported in $B(1)$, $\Theta \geq 0.$  Using the differential equation,
we will express $\int_{\mathbb R^n} (u(x,t_0)-v(x,t_0)) \Theta(x) dx$ in terms of several other quantities.  By appropriately choosing certain parameters, we will show
how to make each of these quantities as small as desired.  Some of the choices of  parameters will
depend on the choices of other parameters, that is, the order in which they are chosen is critical.  At the end of the proof, we explain
the exact order in which to choose these parameters.

We follow a technique of Oleinik \cite{O}.  Define
\begin{equation*}
c(x,t)=
\begin{cases} \frac{\alpha(u(x,t))-\alpha(v(x,t))}{u(x,t)-v(x,t)}  & \hbox{ \ if \  } u(x,t) \neq v(x,t)  \\
              0  & \hbox{ \ if \ } u(x,t)=v(x,t)
\end{cases}  
\end{equation*}
Then $0\leq c(x,t) \leq 1$ a.e. on $\mathbb R^n \times (0,T).$

Suppose $R$ is much larger than $1$, to be chosen precisely later.
For $m=1,2, \dots$ let $c_m(x,t)$ be a regularization of $c(x,t) \vee \frac{1}{m}$ which satisfies
\begin{equation*}
\lim_{m \to \infty} \int_0^T \int_{B(R)} \frac{|c-c_m|^2}{c_m} \, dx dt =0 .
\end{equation*}
For $m=1,2, \dots$ let $\phi_m$ be a classical solution of
\begin{equation*}
\begin{aligned}
\phi_t +c_m \Delta \phi &=0  &\hbox{ \ \ on \ } B(R) \times (0,t_0), \\
\phi(x, t_0) &=\Theta(x),  & x \in B(R), \\
\phi(x,t)&=0  &(x,t) \in \partial B(R) \times (0,t_0].
\end{aligned}
\end{equation*}
By the maximum principle, $\|\phi_m\|_{\infty} \leq \|\Theta\|_{\infty}.$
Suppose $0 < \delta <t_0,$ where $\delta$ is to be chosen later.  
Then by Lemma \ref{representingformula}, we may write
\begin{equation*}
\begin{aligned}
&\int_{B(R)} (u(x,t_0)-v(x,t_0)) \Theta(x) dx =\int_{B(R)} (u(x, \delta) - v(x, \delta)) \phi_m(x, \delta) dx \\
& - \int_{\partial B(R)}\int_{\delta}^{t_0} \frac{\partial \phi_m}{\partial n} \left[ \alpha(u)-\alpha(v) \right] dt d\sigma 
+ \int_{B(R)}\int_{\delta}^{t_0}  \Delta\phi_m \left[ \alpha(u)- \alpha(v) \right] 
+ \frac{\partial \phi_m} {\partial t} \left[u-v \right] dt dx \\
&= I + II + III,
\end{aligned}
\end{equation*}
each of which we investigate separately.

To estimate III, we use $-c_m \Delta \phi_m =  \frac{\partial}{\partial t} \phi_m$ to obtain
\begin{equation*}
\begin{aligned}
|III|&=\left| \int_{B(R)} \int_{\delta}^{t_0}  \Delta \phi_m \left[\alpha(u) - \alpha(v) \right] 
  -c_m  \Delta \phi_m \left[ u-v \right] dt dx \right| \\
&= \left| \int_{B(R)} \int_{\delta}^{t_0}  \Delta \phi_m  c \left[u-v \right]  -c_m  \Delta \phi_m \left[ u-v \right] dt dx \right| \\
&\leq \|u-v\|_{L^{\infty}(B(R) \times (\delta, t_0])}  \int_{B(R)} \int_\delta^{t_0}  | \Delta \phi_m| |c-c_m| dt dx  \\
& \leq \|u-v\|_{L^{\infty}(B(R) \times (\delta, t_0])} \left( \int_{B(R)} \int_{\delta }^{t_0}  c_m |\Delta \phi_m|^2 dt dx \right)^{\frac{1}{2}} \left(  \int_{B(R)} \int_{\delta}^{t_0} \frac{|c-c_m|^2}{c_m} 
dt dx \right)^{\frac{1}{2}}.
\end{aligned}
\end{equation*}
Multiply the equation $\frac{\partial}{\partial t} \phi_m + c_m \Delta \phi_m=0$ by $\Delta \phi_m$ and integrate by parts to obtain
\begin{equation}\label{energyest}
\frac{1}{2} \int_{B(R)} |\nabla \phi_m(x, \delta)|^2 dx + \int_{\delta}^{t_0} \int_{B(R)} c_m |\Delta \phi_m|^2 dx dt  
                            =\frac{1}{2}\int_{B(R)} |\nabla \Theta(x, t_0)|^2 dx 
\end{equation}
and thus $III \to 0 $ as $m \to \infty$ (for $R$ and $\delta$  fixed).   

To estimate $II$, we use Lemma \ref{normalderivative}, with $t$ replaced by $t_0-t$:  
\begin{equation*}
\begin{aligned}
II &\leq \int_{\partial B(R)} \int_{\delta}^{t_0}  \left|\frac{\partial\phi_m(x,t)}{\partial n} \right| \left(|\alpha(u(x,t))|+|\alpha(v(x,t))|\right) dt d\sigma \\
   &\leq \int_{\partial B(R)} \int_{\delta}^{t_0}\exp\left(\frac{-R^2}{8(t_0-t)}\right) \exp(c|x|^2)\left(|\alpha(u)|+|\alpha(v)|\right) \exp(-c|x^2|) dt d\sigma \\
   &\leq \exp\left(cR^2- \frac{R^2}{8(t_0-t)}\right)\int_{\partial B(R)}\int_0^T \left(|\alpha(u)|+|\alpha(v)|\right) 
\exp(-c|x^2|) dt d\sigma
\end{aligned}
\end{equation*} 
Since $\int_{\mathbb R^n} \int_0^T  \left(|\alpha(u)|+|\alpha(v)|\right) \exp(-c|x^2|) dt dx = M < \infty,$
there exists a set $G \subset [0, \infty)$ with the property that $G \cap [L, \infty)$ has positive measure for any
$L>0$, such that for $R\in G,$  
\begin{equation*}
\int_{\partial B(R) } \int_0^T  \left(|\alpha(u)|+|\alpha(v)|\right) \exp(-c|x^2|) dt d\sigma  \leq \frac{M}{R^n}.
\end{equation*}
Furthermore, $\exp\left(cR^2- \frac{R^2}{8(t_0-t)}\right) \to 0$ as $R \to \infty$ since $t_0 < \frac{1}{8c},$ and hence, $t_0-t <
\frac{1}{8c}.$ Therefore, we may choose $R$ sufficiently large (in $G$), and independent of $\delta$ and $m$ so that $II$ is as small 
as 
desired.

We would now like to estimate  $I=\int_{B(R)} (u(x, \delta)-v(x, \delta)) \phi_m(x,\delta) dx.$  
We use a variation of the strategy so far.  Let $q_m(x,t)$ be the solution to
\begin{equation*}
\begin{aligned}
q_t + \Delta q &=0  & \hbox{ \ on \ } B(R) \times (-\infty, \delta), \\
q(x,\delta)&= \phi_m(x, \delta) ,  & x \in B(R), \\
q(x,t)&=0  &\hbox{\ on \ } \partial B(R) \times (-\infty, \delta).
\end{aligned}
\end{equation*} 
Let $0 < \gamma < \delta.$ Then by Lemma \ref{representingformula},
\begin{equation*}
\begin{aligned}
&\int_{B(R)} (u(x,\delta) -v(x,\delta)) \phi_m(x,\delta) dx = \int_{B(R)}\left( u(x,\gamma)-v(x,\gamma)\right) q_m(x,\gamma) dx \\
& \ - \int_{\partial B(R)} \int_{\gamma}^{\delta} \frac{\partial q_m}{\partial n}[\alpha(u)-\alpha(v)]dt d\sigma
+ \int_{B(R)} \int_{\gamma}^{\delta} \Delta q_m [\alpha(u)-\alpha(v)]+ \frac{\partial q_m}{\partial t}[u-v] dtdx \\
&= I_1 + I_2 + I_3.
\end{aligned}
\end{equation*}
We first estimate $I_3$ in a similar fashion to our estimation of $III$. Since $\frac{\partial q_m}{\partial t} =- \Delta q_m,$
\begin{equation*}
\begin{aligned}
|I_3| &= \left| \int_{B(R)} \int_{\gamma}^{\delta} \Delta q_m \left[ \alpha(u) - \alpha(v) - (u-v)\right] 
dtdx \right| 
\leq 2 \int_{B(R)} \int^{\delta}_{\gamma} |\Delta q_m| dtdx \\
     & \leq 2 \left(\int_{B(R)} \int_{\gamma}^{\delta} 1 dt dx \right)^{\frac{1}{2}} \left( \int_{B(R)} \int_{\gamma}^{\delta} |\Delta 
q_m|^2 dt dx\right)^{\frac{1}{2}} 
           = 2 \sqrt{|B(R)|} \sqrt{\delta} \left(\int_{B(R)} \int_{\gamma}^{\delta} |\Delta q_m|^2 dt dx \right)^{\frac{1}{2}}.
\end{aligned}
\end{equation*}
Similar to \eqref{energyest},
multiply the equation $0=\frac{\partial q_m}{\partial t}  + \Delta q_m $ by $\Delta q_m$ and integrate by parts to obtain
\begin{equation*}
\begin{aligned}
&\frac{1}{2} \int_{B(R)} |\nabla q_m (x,\gamma) |^2 dx + \int_{B(R)}\int_{\gamma}^{\delta} |\Delta q_m|^2 dx dt = \frac{1}{2} \int_{B(R)} |\nabla q_m (x, \delta )|^2 dx  \\
&=\frac{1}{2} \int_{B(R)} |\nabla \phi_m (x, \delta)|^2 dx \leq \frac{1}{2} \int_{B(R)}|\nabla \Theta(x,t_0)|^2 dx
\end{aligned}
\end{equation*}
where we have used \eqref{energyest}.  
Thus, $|I_3| \leq C \sqrt{|B(R)|} \sqrt{\delta}.$

Our estimation of  $I_2$ is similar to that of $II$. 
We claim $|\frac{\partial q_m (x,t)}{\partial n}| \leq C\exp(\frac{-R^2}{8(t_0-t)})$ for $(x,t) \in \partial B(R)\times (0, \delta).$
This follows from a variation of Lemma \ref{normalderivative}. Consider that lemma and its proof.  Let $h$ be a solution of the equation
as in the statement of the lemma; then as concluded there, $|\frac{\partial h}{\partial n}(x,t)| \leq C \exp(\frac{-R^2}{8t})$ for 
$|x|=R$ and  
$0<t<T.$ Suppose however, $0<T_1<T$ is fixed, and define $r(x,t)$  as the solution of $\Delta r=r_t$ on $B(R) \times (T_1,T),$
$r(x,t)=0$ on $\partial B(R) \times [T_1, T),$ and $r(x,T_1)=h(x, T_1),$ for $x \in B(R).$  Then, following the proof of Lemma 
\ref{normalderivative} (with the notation there) we have that $h(x,t) \leq W(x,t)$ on $D \times (0,T),$ so in particular, 
$r(x,T_1)= h(x,T_1) \leq 
W(x,T_1)$ for $x \in D.$  Furthermore, reasoning exactly as before, $r(x,t) \leq W(x,t)$ on $\partial D \times (T_1,T).$ So by the 
maximum 
principle, $r(x,t) \leq W(x,t)$ on $D \times (T_1,T).$ Continuing as in the proof of Lemma \ref{normalderivative} we conclude
$|\frac{\partial r(R,0,\dots,0,t)}{\partial n}| \leq C\exp(\frac{-R^2}{8t}),$ 
and hence by rotation, $|\frac{\partial r(x,t)}{\partial n}| \leq C\exp(\frac{-R^2}{8t})$ for all $(x,t) \in \partial B(R) \times 
(T_1,T).$
Applying this to $q_m(x,t_0-t)$ yields the desired estimate 
$|\frac{\partial q_m (x,t)}{\partial n}| \leq C\exp(\frac{-R^2}{8(t_0 -t)})$ for $(x,t) \in \partial B(R)\times (0, \delta).$

With this estimate in hand, the estimation of $I_2$  follows exactly the same steps as the estimation of $II.$ Consequently, we may 
choose $R$ sufficiently large, independent of $\delta,$ $\gamma$ and $m,$ so that $|I_2|$ is as small as desired.

We finally estimate $I_1.$
\begin{equation*}
\begin{aligned}
&\int_{B(R)} (u(x, \gamma) -v(x, \gamma)) q_m(x,\gamma) dx =   \\
&\int_{B(R)} (u(x, \gamma) -v(x, \gamma)) (q_m(x,\gamma)-q_m(x,0)) dx + \int_{B(R)} (u(x, \gamma) -v(x, \gamma)) q_m(x,0)  dx. 
\end{aligned}
\end{equation*}
If $\gamma$ is small, the next to last integral is small since $\|u( \cdot,t)- v(\cdot, t) \|_{L^1(B(R))}$ are bounded (see the remarks 
after the statement of the theorem)
and the fact that $q_m(x,t) 
\to q_m(x,0)$ uniformly as $t \to 0.$ 
The last integral is small if $\gamma$ is small by hypothesis.

Let us explain the order in which various constants are chosen. First $R$ should be chosen so that $II$ and $I_2$ are small.
Then $\delta$ should be chosen so that $I_3$ is small. Then $m$ should be chosen so that $III$ is small. 
Then $\gamma$ should be chosen so that $I_1$ is small.
This forces $\int_{B(R)} (u(x,t_0)-v(x,t_0)) \Theta(x) dx$ to be as small as desired, which completes the 
proof.

\section{Further remarks}
  
With only slight modification, the proof of Theorem \ref{maintheorem} can yield a more general statement. 

\begin{theorem}\label{generalization}
Suppose $\alpha : \mathbb R \to \mathbb R$ is nondecreasing and Lipschitz, and that there exists a number $a \geq 0$ so that
$\alpha(u)-au$ is bounded. 
Suppose $u$ and $v$ are  solutions of
\eqref{distsolns}   on $\mathbb R^n \times (0,T)$ which
belong to $L^{\infty}(\mathbb R^n \times (\varepsilon, T))$ for 
every $\varepsilon >0,$ and which satisfy 
 $\int_{\mathbb R^n} \int_0^T (|u(x,t)| + |v(x,t)|) e^{-c|x|^2} dt dx < \infty,$
for some $c>0.$  
If  for every $\varphi \in C^{\infty}(\mathbb R^n)$  of compact 
support we have
$\lim_{t \to 0} \int_{\mathbb R^n} (u(x,t)-v(x,t)) \varphi(x) dx =0,$
then $u(x,t)=v(x,t)$ a.e. on $\mathbb R^n \times (0,T).$
\end{theorem}

Neither Theorems \ref{maintheorem} or \ref{generalization}  allow for $\alpha$ which are only locally Lipschitz  such as $\alpha(u) = 
sgn(u) |u|^m$ in the porous
medium equation.   B\'enilan, Crandall and Pierre \cite {BCP} have shown uniqueness, in the sense of distributions,  
for solutions $u \in C([0,T): L^1_{loc}(\mathbb R^n))$
of the porous medium equation which satisfy a certain growth condition.  But for the porous medium equation, uniqueness for signed solutions with the
initial data taken as measures, remains an open problem.  See Daskalopolous and Kenig \cite{DK} for a discussion of uniqueness results for the porous medium equation and
this open question.

An interesting first step toward  more general uniqueness results would be to show the main theorem in the case when $\alpha$ is only assumed Lipschitz (or even locally
Lipschitz) and nondecreasing. (With possibly different growth conditions on the solutions.)  These are the assumptions in the theorem of Pierre \cite{P} on uniqueness for nonnegative solutions, and the 
theorem of Bouillet \cite{B} on uniqueness for signed solutions
which assumes $L^1$ convergence as $t \to 0.$

\end{document}